\documentclass[reqno,10pt]{article} % font size = 12
\usepackage{amsmath,amsthm,amssymb,amscd,verbatim,framed}
\usepackage{graphicx}
\usepackage{stmaryrd}
\usepackage{graphicx}
\usepackage{epsfig}
\usepackage{stmaryrd}
\newcommand{\ignore}[1]{}

\input{epsf}

\newtheorem{prelem}{{\bf Theorem}}

\newtheorem{theorem}{Theorem}

\newtheorem{conjecture}[theorem]{Conjecture}

\newtheorem{lemma}[theorem]{Lemma}

\newtheorem{remarka}[theorem]{Remark}
\newenvironment{remark}{\begin{remarka}\rm}{\hfill\rule{2mm}{2mm}\end{remarka}}
\newtheorem{examplea}[theorem]{Example}
\newenvironment{example}{\begin{examplea}\rm}{\hfill\rule{2mm}{2mm}\end{examplea}}

\newtheorem{exercisea}[theorem]{Exercise}

\def\Z{\mathbb{Z}}
\def\cS{\mathcal{S}}
\def\N{\mathbb{N}}

\title{On the additive bases problem}

\author{
Hamed Hatami
\thanks{Supported by an NSERC grant.}
\\
School of Computer Science\\
McGill University, Montreal\\
\texttt{hatami@cs.mcgill.ca}
\and
Victoria de Quehen
\\
Department of Mathematics and Statistics\\
McGill University, Montreal\\
\texttt{dequehen@math.mcgill.ca}
}

\begin{document}
\maketitle

\begin{abstract}
%Consider $k \ge c_p \log n$  bases for the vector space $\Z_p^n$. For every $i \le k$, let $A_i$ be the set of all vectors that can be expressed as the sum of some vectors in the $i$-th basis. Alon, Linial, and Meshulam~[Additive bases of vector spaces over prime fields. \emph{J. Combin. Theory Ser. A}, 57(2):203--210, 1991.] showed that $A_1+\ldots+A_n=\Z_p^n.$ In this note, we prove more generally that if $G$ is an Abelian group and $A_1,\ldots,A_k \subseteq G$ satisfy $m A_i=G$ (the $m$-fold sumset), then $A_1+\ldots+A_k=G$ provided that $k \ge c_m \log n$. 
%
%
%Let the multiset $B$ be the union (with repetitions) of $k \ge c_p \log n$  bases for the vector space $\Z_p^n$. 
%Alon, Linial, and Meshulam~[Additive bases of vector spaces over prime fields. \emph{J. Combin. Theory Ser. A}, 57(2):203--210, 1991.] showed that  every element in $\Z_p^n$ is  the sum of the elements of a subset of $B$.  In this note, we prove that if $G$ is an Abelian group and $A_1,\ldots,A_k \subseteq G$ satisfy $m A_i=G$ (the $m$-fold sumset), then $A_1+\ldots+A_k=G$ provided that $k \ge c_m \log n$. 

We prove that if $G$ is an Abelian group and $A_1,\ldots,A_k \subseteq G$ satisfy $m A_i=G$ (the $m$-fold sumset), then $A_1+\ldots+A_k=G$ provided that $k \ge c_m \log n$. This generalizes a result of Alon, Linial, and Meshulam~[Additive bases of vector spaces over prime fields. \emph{J. Combin. Theory Ser. A}, 57(2):203--210, 1991]  regarding the so called additive bases. 

\end{abstract}

\section{Introduction}

Let $p$ be a fixed prime, and let $\Z_p^n$ denote the $n$-dimensional vector space over the field $\Z_p$. Given a multiset $B$ with elements from $\Z_p^n$, let 
$\cS(B) = \left\{\left. \sum_{b \in S} b \ \right| \ S \subseteq B \right\}$. The set $B$ is called an \emph{additive basis} if $\cS(B)=\Z_p^n$. 

Jaeger, Linial, Payan, and Tarzi~\cite{MR1186753} made the following conjecture and showed that if true, it would provide a beautiful generalization of many important results regarding nowhere zero flows. In particular the case $p=3$ would imply the weak $3$-flow conjecture, which has been proven only recently by Thomassen~\cite{MR2885433}.

\begin{conjecture}\cite{MR1186753}
\label{conj:JLPT}
For every prime $p$, there exists a constant $k_p$ such that the union (with repetitions) of any $k_p$ bases for $\Z_p^n$ forms an additive basis.
\end{conjecture}

Let us denote by $k_p(n)$ the smallest $k \in \N$ such that  the union  of any $k$ bases for $\Z_p^n$ forms an additive basis.  In~\cite{MR1111557} two different proofs are given to show that $k_p(n)\le c_p \log n$, where here and throughout the paper the logarithms are in base $2$. The first proof is based on exponential sums and yields the bound $k_p(n) \le 1+(p^2/2) \log 2pn$, and the second proof is based on an algebraic method and yields $k_p(n) \le (p-1) \log n + p-2$. As it is observed in~\cite{MR1111557}, it is easy to construct examples showing that $k_p(n) \ge p$, and in fact, to the best of our knowledge, it is quite possible that $k_p(n)=p$.

Let $G$ be an Abelian group, and for $A,B \subseteq G$, define the sumset $A+B=\{a+b \ | a \in A, b \in B\}$. For $A \subseteq G$ and $m \in \N$, let $mA=A+\ldots+A$ denote the $m$-fold sumset of $A$. Note that for a  basis $B$ of $\Z_p^n$, we have $(p-1) \cS(B) = \Z_p^n$.  On the other hand if $B= B_1 \cup \ldots \cup B_k$ is a union with repetitions of $k$ bases for $\Z_p^n$, then $\cS(B)=\cS(B_1)+\ldots+\cS(B_k)$. Hence  Theorem~\ref{thm:main} below is a generalization of the above mentioned theorem of Alon \emph{et al}~\cite{MR1111557}.

\begin{theorem}[Main theorem]
\label{thm:main}
Let $G$ be a finite Abelian group and suppose that  $A_1,\ldots,A_{2K} \subseteq G$ satisfy $m A_i=G$ for all $1\le i \le 2K$ where $K \ge  m \ln \log(|G|)$.  Then $A_1+\ldots+A_{2K}=G$. Moreover for $m=2$, it suffices to have  $K \ge \log \log(|G|)$.
\end{theorem}

We present the proof of Theorem~\ref{thm:main} in Section~\ref{sec:proof}. While it is quite possible that Conjecture~\ref{conj:JLPT} is true, the following example shows that its generalization, Theorem~\ref{thm:main}, cannot be improved beyond $\Theta(\log \log |G|)$ even when $m=2$.

\begin{example}
\label{ex:ConstructionLog}
Let $n=2^k$ and for $i=1,\ldots,k$, let $C_i \subseteq  \Z_p^{2^i}$ be the set of vectors in $\Z_p^{2^i} \setminus \{\vec{0}\}$ in which the first half or the second half (but not both)  of the coordinates are all $0$'s. Note that $C_i+C_i=\Z_p^{2^i}$. Define $A_0=(\Z_p \setminus \{0\})^{2^k}$ and for $i=1,\ldots,k$, let 
$$A_i=\underbrace{C_i \times \ldots \times C_i}_{2^{k-i}} \subseteq \Z_p^n.$$ It follows from $C_i+C_i=\Z_p^{2^i}$ that $A_i+A_i=\Z_p^n$. On the other hand a simple induction shows that for $j \le k$, 
$$A_0+\ldots+A_j= (\Z_p^{2^j} \setminus \{\vec{0} \})^{2^{k-j}} \neq \Z_p^n.$$
\end{example}

\begin{remark}
Theorem~\ref{thm:main} in particular implies that $k_p(n) \le 2 (p-1) \ln n + 2 (p-1) \ln \log p$, and  $k_3(n) \le 2 \log n + 2$. Note that for $p >3$, the algebraic proof of~\cite{MR1111557} provides a slightly better constant, however unlike the theorem of~\cite{MR1111557}, Theorem~\ref{thm:main} can be applied to the case where $p$ is not necessarily a prime. 
\end{remark}

\section{Proof of Theorem~\ref{thm:main}}
\label{sec:proof}

The proof is based on the Pl\"unnecke-Ruzsa inequality. 

\begin{lemma}[Pl\"unnecke-Ruzsa]
\label{lem:PlunneckeRuzsa}
If $A,B$ are finite sets in an Abelian group satisfying $|A+B| \le \alpha |B|$, then 
$$|k A | \le \alpha^{k} |B|,$$
provided that $k > 1$. 
\end{lemma}

Next we present the proof  of Theorem~\ref{thm:main}. For $2 \le i\le K$, substituting $k=m$, $A=A_{i}$ and $B=A_1+\ldots+A_{i-1}$ in Lemma~\ref{lem:PlunneckeRuzsa}, we obtain 
$$|G|= |mA_i| \le \left(\frac{|A_1+\ldots+A_{i-1}+A_i|}{|A_1+\ldots+A_{i-1}|} \right)^m |A_1+\ldots+A_{i-1}|,$$
which simplifies to
$$|G|^{1/m}  |A_1+\ldots+A_{i-1}|^{\frac{m-1}{m}} \le |A_1+\ldots+A_{i-1}+A_i|.$$
Consequently 
$$|G|^{1-\lambda}  |A_1|^{\lambda} \le |A_1+\ldots+A_K|,$$
where $\lambda=\left(\frac{m-1}{m}\right)^K$. Since $K \ge  m \ln \log |G|$, we have $\lambda=\left(\frac{m-1}{m}\right)^K < e^{-K/m} \le 1/\log |G|$, and thus 
$|G|^\lambda < 2$ and   $|G|/2 <|A_1+\ldots+A_K|$. Similarly we obtain 
$$|G|/2 <  |A_{K+1}+\ldots+A_{2K}|.$$
Since $A+B=G$ if $|A|,|B| > |G|/2$, we conclude  
$$A_1+\ldots+A_{2K}=G.$$

Finally note that for $m=2$, we have $\lambda=2^{-K}$, and thus to obtain $|G|/2 < |G|^{1-\lambda}  |A_1|^{\lambda}$, it suffices to have $K \ge \log \log|G|$.

\section{Quasi-random Groups}

While Example~\ref{ex:ConstructionLog} shows that the bound of $\Theta(\log \log |G|)$ is essential in Theorem~\ref{thm:main}, for certain non-Abelian groups, it is possible to achieve the constant bound similar to what is conjectured in Conjecture~\ref{conj:JLPT}.  A finite group $G$ is called $D$-quasirandom if all non-trivial unitary representations of $G$ have dimension at least $D$. The terminology ``quasirandom group'' was introduced explicitly by Gowers in the fundamental paper~\cite{MR2410393} where he showed that the dense Cayley graphs in quasirandom groups are quasirandom graphs in the sense of Chung, Graham, and Wilson~\cite{MR1054011}. The group $\mathrm{SL}_2(\Z_p)$ is an example of a highly quasirandom group. The so called Frobenius lemma says that $\mathrm{SL}_2(\Z_p)$ is $(p-1)/2$-quasirandom. This has to be compared to the cardinality of this group,  $|\mathrm{SL}_2(\Z_p)|=p^3-p$. The basic fact that we will use about the quasirandom groups is the following theorem of Gowers (See also~\cite[Exercise 3.1.1]{MR3309986}). 

\begin{theorem}[{\cite{MR2410393}}]
\label{thm:GowersQuasirandom}
Let $G$ be a $D$-quasirandom finite group.  Then every $A,B,C \subseteq G$ with $|A||B||C| > |G|^3 /D$ satisfy $ABC=G$. 
\end{theorem}

We will also need the noncommutative version of Ruzsa's inequality.  

\begin{lemma}[Ruzsa inequality]
\label{lem:Ruzsa}
Let $A,B,C \subseteq G$  be finite subsets of a group $G$. Then 
$$|A C^{-1}| \le \frac{|A B^{-1}||B C^{-1}|}{|B|}.$$
\end{lemma}
\begin{proof}
The claims follows immediately from fact that by the identity $ac^{-1}=a b^{-1} bc^{-1}$,  every element $ac^{-1}$ in $AC^{-1}$ has at least $|B|$ distinct representations of the from $xy$ with 
$(x,y) \in (AB^{-1}) \times (BC^{-1})$.
\end{proof}

Finally we can state the analogue of  Theorem~\ref{thm:main} for quasi-random groups. 

\begin{theorem}
Let $G$ be a $|G|^{\delta}$-quasirandom finite group for some $\delta > 0$. If the sets $A_1,\ldots,A_{K} \subseteq G$ satisfy $A_i A_i^{-1}=G$ for all $1\le i \le K$ where $K > \log(3/\delta)$.  Then $A_1\cdot \ldots \cdot A_{3K}=G$.
\end{theorem}
\begin{proof}
Obviously $|A_1| \ge |G|^{1/2}$. For $2 \le i\le K$, substituting $A=C=A_i^{-1}$ and $B=A_1 \cdot \ldots \cdot A_{i-1}$ in Lemma~\ref{lem:Ruzsa}, we obtain 
$$\sqrt{|G||A_1 \cdot \ldots \cdot A_{i-1}|} \le |A_1 \cdot \ldots \cdot A_{i}|,$$
which in turn shows 
$$|G|^{1-2^{-K}} \le |A_1 \cdot \ldots \cdot A_{K}|.$$
Since $K > \log(3/\delta)$, we have 
$$|G| |G|^{-\delta/3} < |A_1\cdot \ldots \cdot A_K|.$$
We obtain a similar bound for $|A_{K+1} \cdot \ldots \cdot A_{2K}|$ and $|A_{2K+1} \cdot \ldots \cdot A_{3K}|$, and the result follows from Theorem~\ref{thm:GowersQuasirandom}.
\end{proof}
\begin{remark}
Note that in particular for $G=\mathrm{SL}_2(\Z_p)$, if $p \ge 7$, and $A_1,\ldots,A_{12} \subseteq G$ satisfy $A_i A_i^{-1}=G$, then $A_1 \ldots A_{12}=G$. 
\end{remark}

\section*{Acknowledgements.}
We would like to thank  Kaave Hosseini, Nati Linial, and Shachar Lovett for valuable discussions about this problem.

\bibliographystyle{alpha}
\bibliography{note}

\end{document}